\documentclass[sn-mathphys]{sn-jnl}

\jyear{2021}%

\theoremstyle{thmstyleone}%
\newtheorem{theorem}{Theorem}
\newtheorem{proposition}[theorem]{Proposition}%

\theoremstyle{thmstyletwo}%
\newtheorem{lemma}{Lemma}%

\theoremstyle{thmstylethree}%

\begin{document}

\newcommand{\hide}[1]{}
\newcommand{\tbox}[1]{\mbox{\tiny #1}}
\newcommand{\half}{\mbox{\small $\frac{1}{2}$}}
\newcommand{\sinc}{\mbox{sinc}}
\newcommand{\const}{\mbox{const}}
\newcommand{\trc}{\mbox{trace}}
\newcommand{\intt}{\int\!\!\!\!\int }
\newcommand{\ointt}{\int\!\!\!\!\int\!\!\!\!\!\circ\ }
\newcommand{\eexp}{\mbox{e}^}
\newcommand{\bra}{\left\langle}
\newcommand{\ket}{\right\rangle}
\newcommand{\EPS} {\mbox{\LARGE $\epsilon$}}
\newcommand{\ar}{\mathsf r}
\newcommand{\bmsf}[1]{\bm{\mathsf{#1}}}
\newcommand{\mpg}[2][1.0\hsize]{\begin{minipage}[b]{#1}{#2}\end{minipage}}

\newcommand{\CC}{\mathbb{C}}
\newcommand{\NN}{\mathbb{N}}
\newcommand{\PP}{\mathbb{P}}
\newcommand{\RR}{\mathbb{R}}
\newcommand{\QQ}{\mathbb{Q}}
\newcommand{\ZZ}{\mathbb{Z}}

\title[Multiplicative topological indices: Analytical and statistical properties]{Multiplicative topological indices: Analytical properties and application to random networks}


\author[1]{\fnm{R.} \sur{Aguilar-S\'anchez}}\email{ras747698@gmail.com}

\author[2]{\fnm{J. A.} \sur{M\'endez-Berm\'udez}}\email{jmendezb@ifuap.buap.mx}
\equalcont{These authors contributed equally to this work.}

\author[3]{\fnm{Jos\'e M.} \sur{Rodr\'iguez}}\email{jomaro@math.uc3m.es}
\equalcont{These authors contributed equally to this work.}

\author*[4]{\fnm{Jos\'e M.} \sur{Sigarreta}}\email{josemariasigarretaalmira@hotmail.com}
\equalcont{These authors contributed equally to this work.}

\affil[1]{\orgdiv{Facultad de Ciencias Qu\'imicas}, \orgname{Benem\'erita Universidad
Aut\'onoma de Puebla}, \city{Puebla}, \postcode{72570}, \state{Puebla}, \country{Mexico}}

\affil[2]{\orgdiv{Instituto de F\'{\i}sica}, \orgname{Benem\'erita Universidad
Aut\'onoma de Puebla}, \orgaddress{\street{Apartado Postal J-48}, \city{Puebla}, \postcode{72570}, \state{Puebla}, \country{Mexico}}}

\affil[3]{\orgdiv{Departamento de Matem\'aticas}, \orgname{Universidad Carlos III de Madrid}, \orgaddress{\street{Avenida de la Universidad 30}, \city{Legan\'es}, \postcode{28911}, \state{Madrid}, \country{Spain}}}

\affil[4]{\orgdiv{Facultad de Matem\'aticas}, \orgname{Universidad Aut\'onoma de Guerrero}, \orgaddress{\street{Carlos E. Adame No.54 Col. Garita}, \city{Acapulco}, \postcode{610101}, \state{Guerrero}, \country{Mexico}}}


\abstract{We make use of multiplicative degree-based topological indices $X_\Pi(G)$ to perform a detailed
analytical and statistical study of random networks $G=(V(G),E(G))$.
We consider two classes of indices: $X_\Pi(G) = \prod_{u \in V(G)} F_V(d_u)$ and
$X_\Pi(G) = \prod_{uv \in E(G)} F_E(d_u,d_v)$,
where $uv$ denotes the edge of $G$ connecting the vertices $u$ and $v$, $d_u$ is the degree
of the vertex $u$, and $F_V(x)$ and $F_E(x,y)$ are functions of the vertex degrees.
Specifically, we find analytical inequalities involving these multiplicative indices.
Also, we apply $X_\Pi(G)$ on three models of random networks: Erd\"os-R\'enyi networks,
random geometric graphs, and bipartite random networks.
We show that $\left< \ln X_\Pi(G) \right>$, normalized to the order of the network, scale with
the corresponding average degree; here $\left< \cdot \right>$ denotes the average over an
ensemble of random networks. }




\maketitle

\section{Introduction}

We can identify two types of graph invariants which are currently been studied in
chemical graph theory, namely
\begin{equation}
\begin{aligned}
& X_\Sigma(G) = X_{\Sigma,F_V}(G) = \sum_{u \in V(G)} F_V(d_u) \qquad \mbox{or} \qquad
\\
& X_\Sigma(G) = X_{\Sigma,F_E}(G) = \sum_{uv \in E(G)} F_E(d_u,d_v)
\end{aligned}
\label{TI}
\end{equation}
and
\begin{equation}
\begin{aligned}
& X_\Pi(G) = X_{\Pi,F_V}(G) = \prod_{u \in V(G)} F_V(d_u) \qquad \mbox{or} \qquad
\\
& X_\Pi(G) = X_{\Pi,F_E}(G) = \prod_{uv \in E(G)} F_E(d_u,d_v) \, .
\end{aligned}
\label{MTI}
\end{equation}
Here $uv$ denotes the edge of the graph $G=(V(G),E(G))$ connecting the vertices $u$ and $v$,
$d_u$ is the degree of the vertex $u$, and $F_V(x)$ and $F_V(x,y)$ are appropriate chosen functions,
see e.g.~\cite{G13}.
While both $X_\Sigma(G)$ and $X_\Pi(G)$ are referred as topological indices in the literature, to make
a distinction between them, here we name $X_\Pi(G)$ as {\it multiplicative} topological indices (MTIs).

Among the vast amount of topological indices of the form given by Eq.~(\ref{TI}) we can refer to
some prominent examples:
the first and second Zagreb indices~\cite{G72}
\begin{equation}
M_1(G) = \sum_{u\in V(G)} d_u^2 = \sum_{uv\in E(G)} d_u + d_v
\label{M1}
\end{equation}
and
\begin{equation}
M_2(G) = \sum_{uv\in E(G)} d_ud_v ,
\label{M2}
\end{equation}
respectively, the Randi\'c connectivity index~\cite{R75}
\begin{equation}
R(G) = \sum_{uv\in E(G)} \frac{1}{\sqrt{d_ud_v}} ,
\label{R}
\end{equation}
the harmonic index~\cite{F87}
\begin{equation}
H(G) = \sum_{uv\in E(G)} \frac{2}{d_u + d_v} ,
\label{H}
\end{equation}
the sum-connectivity index~\cite{ZT46}
\begin{equation}
\chi(G) = \sum_{uv\in E(G)} \frac{1}{\sqrt{d_u+d_v}} ,
\label{X}
\end{equation}
and the inverse degree index~\cite{F87}
\begin{equation}
ID(G) = \sum_{u\in V(G)} \frac{1}{d_u} = \sum_{uv\in E(G)} \left( \frac{1}{d_u^2} + \frac{1}{d_v^2} \right) .
\label{ID}
\end{equation}

In fact, within a statistical approach on random networks, it has been recently shown that the average
values of indices of the type $X_\Sigma(G)$, normalized to the order of the network $n$, scale with the
average degree $\left< k \right>$; see e.g.~\cite{MMRS20,AHMS20,MMRS21}.
That is, $\left< X_\Sigma(G) \right>/n$ is a function of $\left< k \right>$ only.
Moreover, it was also found that $\left< X_\Sigma(G) \right>$, for indices like $R(G)$ and $H(G)$,
is highly correlated with the Shannon entropy of the eigenvectors of the adjacency matrix of random
networks~\cite{AMRS20}. This is a notable result because it puts forward the application of topological
indices beyond mathematical chemistry. Specifically, given que equivalence of the Hamiltonian
of a tight-binding network (in the proper setup) and the corresponding network adjacency matrix,
either $\left< R(G) \right>$ or $\left< H(G) \right>$ could be used to determine the eigenvector
localization and delocalization regimes which in turn determine the insulator and metallic regimes
of quantum transport.

Therefore, motivated by their potential applications, in this work we study MTIs, see Eq.~(\ref{MTI}), on
random networks.

Indeed, several MTIs have been already deeply analyzed in the literature (see e.g. the work of V. R. Kulli), 
among the most relevant we have to mention the Narumi-Katayama index~\cite{NK84}
\begin{equation}
NK(G) = \prod_{u\in V(G)} d_u ,
\label{NK}
\end{equation}
and multiplicative versions of the Zagreb indices~\cite{TC1084}:
\begin{equation}
\Pi_1(G) = \prod_{u\in V(G)} d_u^2 ,
\label{P1}
\end{equation}
\begin{equation}
\Pi_2(G) = \prod_{uv\in E(G)} d_ud_v ,
\label{P2}
\end{equation}
and
\begin{equation}
\Pi_1^*(G) = \prod_{uv\in E(G)} d_u + d_v .
\label{P1*}
\end{equation}
In addition to the MTIs of Eqs.~(\ref{NK}-\ref{P1*}), we also consider multiplicative versions of the
indices in Eqs.~(\ref{R}-\ref{ID}):
the multiplicative Randi\'c connectivity index
\begin{equation}
R_\Pi(G) = \prod_{uv\in E(G)} \frac{1}{\sqrt{d_ud_v}} ,
\label{mR}
\end{equation}
the multiplicative harmonic index
\begin{equation}
H_\Pi(G) = \prod_{uv\in E(G)} \frac{2}{d_u + d_v} ,
\label{mH}
\end{equation}
the multiplicative sum-connectivity index
\begin{equation}
\chi_\Pi(G) = \prod_{uv\in E(G)} \frac{1}{\sqrt{d_u+d_v}} ,
\label{mX}
\end{equation}
and the multiplicative inverse degree index
\begin{equation}
ID_\Pi(G) = \prod_{uv\in E(G)} \left( \frac{1}{d_u^2} + \frac{1}{d_v^2} \right) .
\label{mID}
\end{equation}

Below we apply the MTIs of Eqs.~(\ref{NK}-\ref{mID}) on three models of random
networks:
Erd\"os-R\'enyi (ER) networks, random geometric (RG) graphs, and bipartite random (BR) networks.
ER networks~\cite{SR51,ER59,ER60} $G_{\tbox{ER}}(n,p)$ are formed by $n$ vertices connected independently
with probability $p \in [0,1]$.
While RG graphs~\cite{DC02,P03} $G_{\tbox{RG}}(n,r)$ consist of $n$ vertices uniformly and independently
distributed on the unit square, where two vertices are connected by an edge if their Euclidean distance is less
or equal than the connection radius $r \in [0,\sqrt{2}]$.
In addition we examine BR networks $G_{\tbox{BR}}(n_1,n_2,p)$ composed by two disjoint sets, set 1 and set 2,
with $n_1$ and $n_2$ vertices each such that there are no adjacent vertices within the same set, being
$n=n_1+n_2$ the total number of vertices in the bipartite network. The vertices of the two sets are connected
randomly with probability $p \in [0,1]$.


The purpose of this work is threefold.
First, from an analytical viewpoint, we find several inequalities that relate multiplicative indices to their additive versions;
second, we want to establish the statistical
analysis of MTIs as a generic tool for studying average properties of random networks; and third, we
perform for the first time (to our knowledge), a scaling study of MTIs on random networks.

\section{Some inequalities involving multiplicative topological indices}

We obtain in this section several analytical inequalities involving general multiplicative topological indices $X_{\Pi}$.
There are many papers studying particular multiplicative topological indices (see, e.g., \cite{GutmanMilovanovic}, \cite{RetiGutman} and the references therein).

The following equalities are direct:
$$
\begin{aligned}
X_{\Pi,F_V}(G)
& = \prod_{u \in V(G)} F_V(d_u)
= e^{\sum_{u \in V(G)} \log F_V(d_u) }
= e^{X_{\Sigma,\log F_V}(G) } ,
\\
X_{\Pi,F_E}(G)
& = \prod_{uv \in E(G)} F_E(d_u,d_v)
= e^{\sum_{uv \in E(G)} \log F_E(d_u,d_v) }
= e^{X_{\Sigma,\log F_E}(G) } .
\end{aligned}
$$

Since the geometric mean is at most the arithmetic mean (by Jensen's inequality), we have the following inequalities.

\begin{proposition} \label{p:Jensen}
Let $G$ be a graph with $n$ vertices and $m$ edges.
Then
$$
\begin{aligned}
X_{\Pi,F_V}(G)^{1/n}
& \le \frac1n \, X_{\Sigma,F_V}(G) ,
\qquad
X_{\Pi,F_E}(G)^{1/m}
\le \frac1m \, X_{\Sigma,F_E}(G).
\end{aligned}
$$
\end{proposition}

In \cite{BQRS} appears the following Jensen-type inequality.

\begin{theorem} \label{t:Jensen3}
Let $\mu$ be a probability measure on the space $X$ and $a \le b$ real constants.
If $f: X \rightarrow [a,b]$ is a measurable function and $\varphi$ is a convex function on $[a,b]$, then
$f$ and $\varphi \circ f$ are $\mu$-integrable functions and
$$
\varphi \Big( a+b - \int _{X} f\,d\mu \Big)
\le \varphi(a) + \varphi(b) - \int _{X}\varphi \circ f\,d\mu .
$$
\end{theorem}

Theorem \ref{t:Jensen3} allows to find the following converse version of Proposition \ref{p:Jensen}.

\begin{theorem} \label{t:Jensen4}
Let $G$ be a graph with $n$ vertices and $m$ edges,
$a_V \le b_V$, $a_E \le b_E$ be real constants and $F_V$, $F_E$ be functions satisfying
$$
a_V \le F_V(d_u) \le b_V
\quad \forall \, u \in V(G),
\qquad
a_E \le F_E(d_u,d_v) \le b_E
\quad \forall \, uv \in E(G).
$$
Then
$$
\begin{aligned}
\frac1n\, X_{\Sigma,F_V}(G)
& \le e^{ a_V} + e^{b_V} - \frac{e^{ a_V+b_V}}{X_{\Pi,F_V}(G)^{1/n}} \,,
\\
\frac1m\, X_{\Sigma,F_E}(G)
& \le e^{ a_E} + e^{b_E} - \frac{e^{ a_E+b_E}}{X_{\Pi,F_E}(G)^{1/m}} \, .
\end{aligned}
$$
\end{theorem}

\begin{proof}
Since $\varphi(x)=e^x$ is a convex function on $\mathbb{R}$,
Theorem \ref{t:Jensen3} gives
$$
\begin{aligned}
e^{ a_V+b_V - \frac1n \sum_{u \in V(G)} \log F_V(d_u) }
& \le e^{ a_V} + e^{b_V} - \frac1n \sum_{u \in V(G)} e^{\log F_V(d_u)},
\\
e^{ a_V+b_V} e^{ - \log (\Pi_{u \in V(G)} F_V(d_u))^{1/n} }
& \le e^{ a_V} + e^{b_V} - \frac1n \sum_{u \in V(G)} F_V(d_u),
\\
\frac{e^{ a_V+b_V}}{X_{\Pi,F_V}(G)^{1/n}}
& \le e^{ a_V} + e^{b_V} - \frac1n\, X_{\Sigma,F_V}(G).
\end{aligned}
$$
In a similar way,
$$
\begin{aligned}
e^{ a_E+b_E - \frac1n \sum_{uv \in E(G)} \log F_E(d_u,d_v) }
& \le e^{ a_E} + e^{b_E} - \frac1n \sum_{uv \in E(G)} e^{\log F_E(d_u,d_v)},
\\
e^{ a_E+b_E} e^{ - \log (\Pi_{uv \in E(G)} F_E(d_u,d_v))^{1/n} }
& \le e^{ a_E} + e^{b_E} - \frac1n \sum_{uv \in E(G)} F_E(d_u,d_v),
\\
\frac{e^{ a_E+b_E}}{X_{\Pi,F_E}(G)^{1/m}}
& \le e^{ a_E} + e^{b_E} - \frac1m\, X_{\Sigma,F_E}(G).
\end{aligned}
$$
\end{proof}

The following Kober's inequalities in \cite{Kober} (see also \cite[Lemma 1]{ZGA}) are useful.

\begin{lemma} \label{l:Kober}
If $a_j> 0$ for $1\le j \le k$, then
$$
\sum_{j=1}^k a_j + k(k - 1) \Big( \prod_{j=1}^k a_j \Big)^{1/k}
\le \Big( \sum_{j=1}^k \sqrt{a_j} \Big)^2
\le (k - 1)\sum_{j=1}^k a_j + k \Big( \prod_{j=1}^k a_j \Big)^{1/k} .
$$
\end{lemma}

Lemma \ref{l:Kober} has the following consequence.

\begin{theorem} \label{t:Kober}
Let $G$ be a graph with $n$ vertices and $m$ edges.
Then
$$
\begin{aligned}
X_{\Sigma,F_V^2}(G) + n(n - 1) X_{\Pi,F_V}(G)^{2/n}
& \le X_{\Sigma,F_V}(G)^2
\le (n - 1)X_{\Sigma,F_V^2}(G) + n X_{\Pi,F_V}(G)^{2/n},
\\
X_{\Sigma,F_E^2}(G) + m(m - 1) X_{\Pi,F_E}(G)^{2/m}
& \le X_{\Sigma,F_E}(G)^2
\le (m - 1)X_{\Sigma,F_E^2}(G) + m X_{\Pi,F_E}(G)^{2/m}.
\end{aligned}
$$
\end{theorem}

\begin{proof}
Lemma \ref{l:Kober} gives
$$
\begin{aligned}
\sum_{u \in V(G)} F_V(d_u)^2 + n(n - 1) \Big( \prod_{u \in V(G)} F_V(d_u)^2 \Big)^{1/n}
& \le \Big( \sum_{u \in V(G)} F_V(d_u) \Big)^2,
\\
X_{\Sigma,F_V^2}(G) + n(n - 1) X_{\Pi,F_V}(G)^{2/n}
& \le X_{\Sigma,F_V}(G)^2,
\end{aligned}
$$
and
$$
\begin{aligned}
\Big( \sum_{u \in V(G)} F_V(d_u) \Big)^2
& \le (n - 1)\sum_{u \in V(G)} F_V(d_u)^2 + n \Big( \prod_{u \in V(G)} F_V(d_u)^2 \Big)^{1/n},
\\
X_{\Sigma,F_V}(G)^2
& \le (n - 1)X_{\Sigma,F_V^2}(G) + n X_{\Pi,F_V}(G)^{2/n} .
\end{aligned}
$$

In a similar way,
$$
\begin{aligned}
\sum_{uv \in E(G)} F_E(d_u,d_v)^2 + m(m - 1) \Big( \prod_{uv \in E(G)} F_E(d_u,d_v)^2 \Big)^{1/m}
& \le \Big( \sum_{uv \in E(G)} F_E(d_u,d_v) \Big)^2,
\\
X_{\Sigma,F_E^2}(G) + m(m - 1) X_{\Pi,F_E}(G)^{2/m}
& \le X_{\Sigma,F_E}(G)^2,
\end{aligned}
$$
and
$$
\begin{aligned}
\Big( \sum_{uv \in E(G)} F_E(d_u,d_v) \Big)^2
& \le (m - 1)\sum_{uv \in E(G)} F_E(d_u,d_v)^2 + m \Big( \prod_{uv \in E(G)} F_E(d_u,d_v)^2 \Big)^{1/m},
\\
X_{\Sigma,F_E}(G)^2
& \le (m - 1)X_{\Sigma,F_E^2}(G) + m X_{\Pi,F_E}(G)^{2/m} .
\end{aligned}
$$
\end{proof}

Let us recall the following known Petrovi\'c inequality \cite{Petrovic}.

\begin{theorem} \label{t:Petrovic0}
Let $\varphi$ be a convex function on $[0,a]$, and $w_1,\dots,w_n \ge 0$.
If $t_1,\dots,t_n \in [0,a]$ satisfy $\sum_{k=1}^n t_kw_k \in (0,a]$, and
$$
\sum_{k=1}^n t_kw_k \ge t_j,
\qquad
j = 1, \dots, n,
$$
then
$$
\sum_{k=1}^{n} \varphi(t_k) w_k
\le \varphi\Big(\sum_{k=1}^n t_kw_k\Big) + \Big(\sum_{k=1}^n w_k-1\Big) \varphi(0) .
$$
\end{theorem}

Theorem \ref{t:Petrovic0} has the following consequence.

\begin{proposition} \label{p:Petrovic0}
Let $\varphi$ be a convex function on $[0,a]$.
If $t_1,\dots,t_n \in [0,a]$ satisfy $\sum_{k=1}^n t_k \in (0,a]$, then
$$
\sum_{k=1}^{n} \varphi(t_k)
\le \varphi\Big(\sum_{k=1}^n t_k \Big) + (n-1) \varphi(0) .
$$
\end{proposition}

Proposition \ref{p:Petrovic0} allows to prove the following inequalities.

\begin{theorem} \label{t:Petrovic}
Let $G$ be a graph with $n$ vertices and $m$ edges.
Then
$$
X_{\Sigma,F_V}(G)
\le X_{\Pi,F_V}(G) + n-1 ,
\quad
X_{\Sigma,F_E}(G)
\le X_{\Pi,F_E}(G) + m-1 .
$$
\end{theorem}

\begin{proof}
Let us consider the convex function $\varphi(x) = e^x$.
Proposition \ref{p:Petrovic0} gives
$$
\begin{aligned}
X_{\Sigma,F_V}(G)
& = \sum_{u \in V(G)} F_V(d_u)
= \sum_{u \in V(G)} e^{\log F_V(d_u) }
\\
& \le e^{\sum_{u \in V(G)} \log F_V(d_u) } + n-1
= e^{\log \Pi_{u \in V(G)} F_V(d_u) } + n-1
\\
& = \Pi_{u \in V(G)} F_V(d_u) + n-1
= X_{\Pi,F_V}(G) + n-1 .
\end{aligned}
$$
In a similar way,
$$
\begin{aligned}
X_{\Sigma,F_E}(G)
& = \sum_{uv \in E(G)} F_E(d_u,d_v)
= \sum_{uv \in E(G)} e^{\log F_E(d_u,d_v) }
\\
& \le e^{\sum_{uv \in E(G)} \log F_E(d_u,d_v) } + m-1
= e^{\log \Pi_{uv \in E(G)} F_E(d_u,d_v) } + m-1
\\
& = \Pi_{uv \in E(G)} F_E(d_u,d_v) + m-1
= X_{\Pi,F_E}(G) + m-1 .
\end{aligned}
$$
\end{proof}

\begin{theorem} \label{t:Petrovic}
Let $G$ be a graph.
Then
$$
X_{\Pi,F_V}(G)
\ge X_{\Sigma,\log F_V}(G) + 1 ,
\qquad
X_{\Pi,F_E}(G)
\ge X_{\Sigma,\log F_E}(G) + 1 .
$$
\end{theorem}

\begin{proof}
The equality $e^x \ge x+1$ holds for every $x \in \RR$.
This inequality gives
$$
\begin{aligned}
X_{\Pi,F_V}(G)
& = e^{\log X_{\Pi,F_V}(G)}
\ge \log X_{\Pi,F_V}(G) +1
\\
& = \log \Big( \Pi_{u \in E(G)} F_V(d_u) \Big) +1
= \sum_{u \in E(G)} \log F_V(d_u) +1
\\
& = X_{\Sigma,\log F_V}(G) + 1 .
\end{aligned}
$$
The same argument gives the second inequality.
\end{proof}

\section{Statistical properties of multiplicative topological indices on random networks}

\subsection{Multiplicative topological indices on Erd\"os-R\'enyi random networks}
\label{ER}

Before computing MTIs on ER random networks we note that in the
dense limit, i.e. when $\left<  d \right>\gg 1$, we can approximate $d_u \approx d_v \approx \left<  d \right>$
in Eqs.~(\ref{NK}-\ref{mID}), with
\begin{equation}
\label{k}
\left< d \right> = (n-1)p .
\end{equation}
Thus, for example, when $np\gg 1$, we can approximate $NK_\Pi(G_{\tbox{ER}})$ as
$$
NK(G_{\tbox{ER}}) = \prod_{u\in V(G)} d_u \approx \prod_{u\in V(G)} \left<  d \right> \approx \left<  d \right>^n ,
$$
which leads us to
$$
\ln NK(G_{\tbox{ER}}) \approx n \ln \left<  d \right>
$$
or
\begin{equation}
\label{NKofGER}
\frac{\ln NK(G_{\tbox{ER}})}{n} \approx \ln \left<  d \right> .
\end{equation}
A similar approximation gives
\begin{equation}
\label{P1ofGER}
\frac{\ln \Pi_1(G_{\tbox{ER}})}{n} \approx 2 \ln \left<  d \right> .
\end{equation}
Also, for $\Pi_2(G_{\tbox{ER}})$ we have that
$$
\Pi_2(G_{\tbox{ER}}) = \prod_{uv\in E(G)} d_ud_v \approx \prod_{uv\in E(G)} \left<  d \right>\left<  d \right> =
\prod_{uv\in E(G)} \left<  d \right>^2 \approx  \left<  d \right>^{n\left<  d \right>} ,
$$
where we have used $\mid E(G_{\tbox{ER}}) \mid=n\left<  d \right>/2$.
Therefore
$$
\ln \Pi_2(G_{\tbox{ER}}) \approx n \left<  d \right> \ln \left<  d \right>
$$
or
\begin{equation}
\label{P2ofGER}
\frac{\ln \Pi_2(G_{\tbox{ER}})}{n} \approx \left<  d \right> \ln \left<  d \right> .
\end{equation}
Following this procedure, for the rest of the MTIs of Eqs.~(\ref{P1*}-\ref{mID}) we get:
\begin{equation}
\label{P1*ofGER}
\frac{\ln \Pi_1^*(G_{\tbox{ER}})}{n} \approx \frac{1}{2} \left<  d \right> \ln (2\left<  d \right>) ,
\end{equation}
\begin{equation}
\frac{\ln R_\Pi(G_{\tbox{ER}})}{n} \approx -\frac{1}{2} \left<  d \right> \ln \left<  d \right> ,
\label{mRofGER}
\end{equation}
\begin{equation}
\label{mHofGER}
\frac{\ln H_\Pi(G_{\tbox{ER}})}{n} \approx -\frac{1}{2} \left<  d \right> \ln \left<  d \right> ,
\end{equation}
\begin{equation}
\frac{\ln \chi_\Pi(G_{\tbox{ER}})}{n} \approx -\frac{\ln 2}{4} \left<  d \right>  - \frac{1}{4}\left<  d \right> \ln \left<  d \right> ,
\label{mXofGER}
\end{equation}
and
\begin{equation}
\frac{\ln ID_\Pi(G_{\tbox{ER}})}{n} \approx \frac{\ln 2}{2} \left<  d \right>  - \left<  d \right> \ln \left<  d \right>.
\label{mIDofGER}
\end{equation}

From the approximate expressions above we note that the logarithm of the MTIs
on ER random networks, normalized to the network size $n$, does not depend on $n$.
That is, in the dense limit, we expect $\ln X_\Pi(G_{\tbox{ER}})/n$ to depend on $\left<  d \right>$ only;
here $X$ represents all the MTIs of Eqs.~(\ref{NK}-\ref{mID}).

\begin{figure}[t!]
\begin{center}
\includegraphics[width=0.75\textwidth]{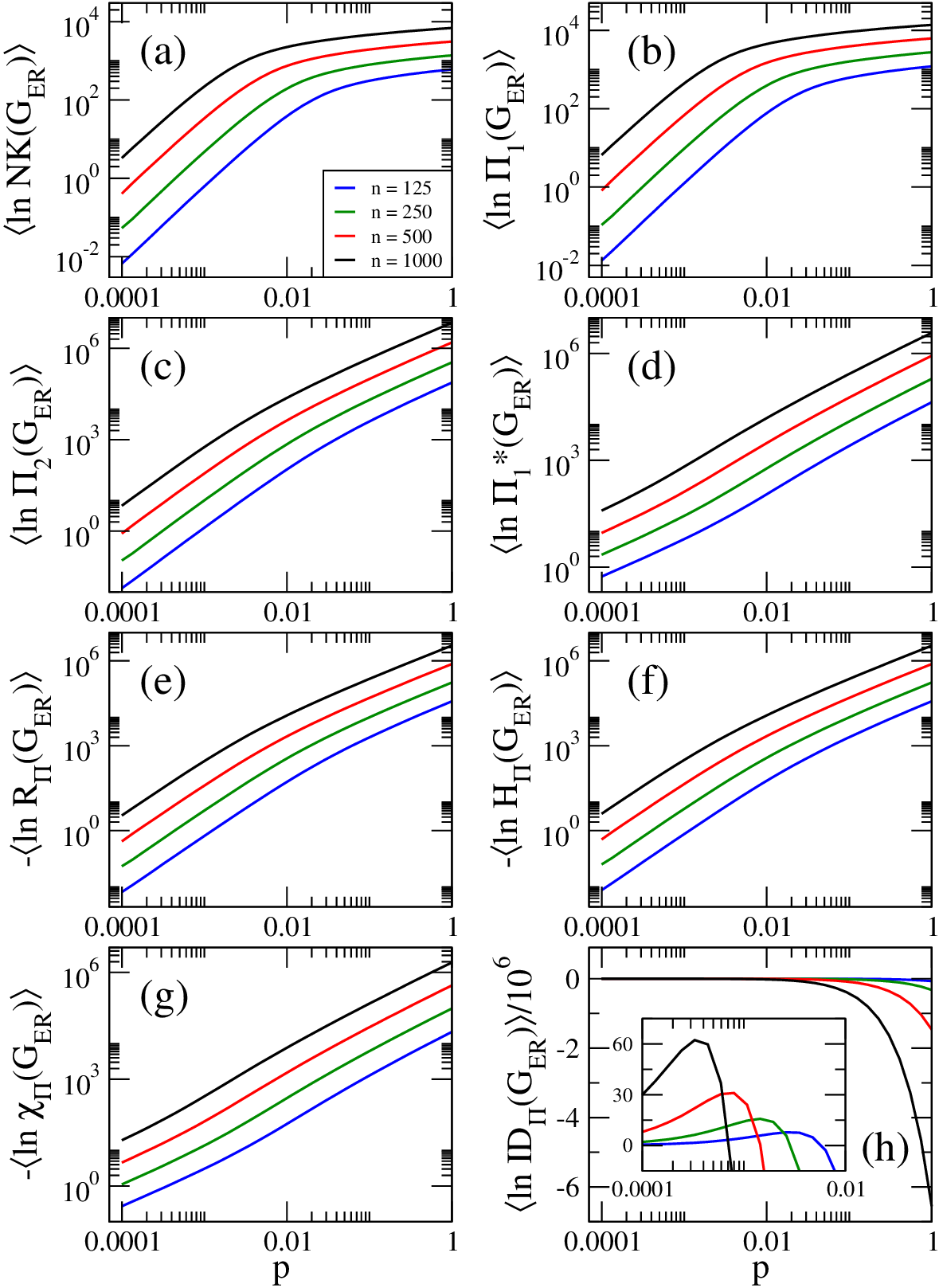}
\caption{\footnotesize{
Average logarithm of the multiplicative topological indices of Eqs.~(\ref{NK}-\ref{mID}) as a function of the
probability $p$ of Erd\"os-R\'enyi networks of size $n$. The inset in (h) is an enlargement for small $p$.
}}
\label{Fig01}
\end{center}
\end{figure}
\begin{figure}[t!]
\begin{center}
\includegraphics[width=0.75\textwidth]{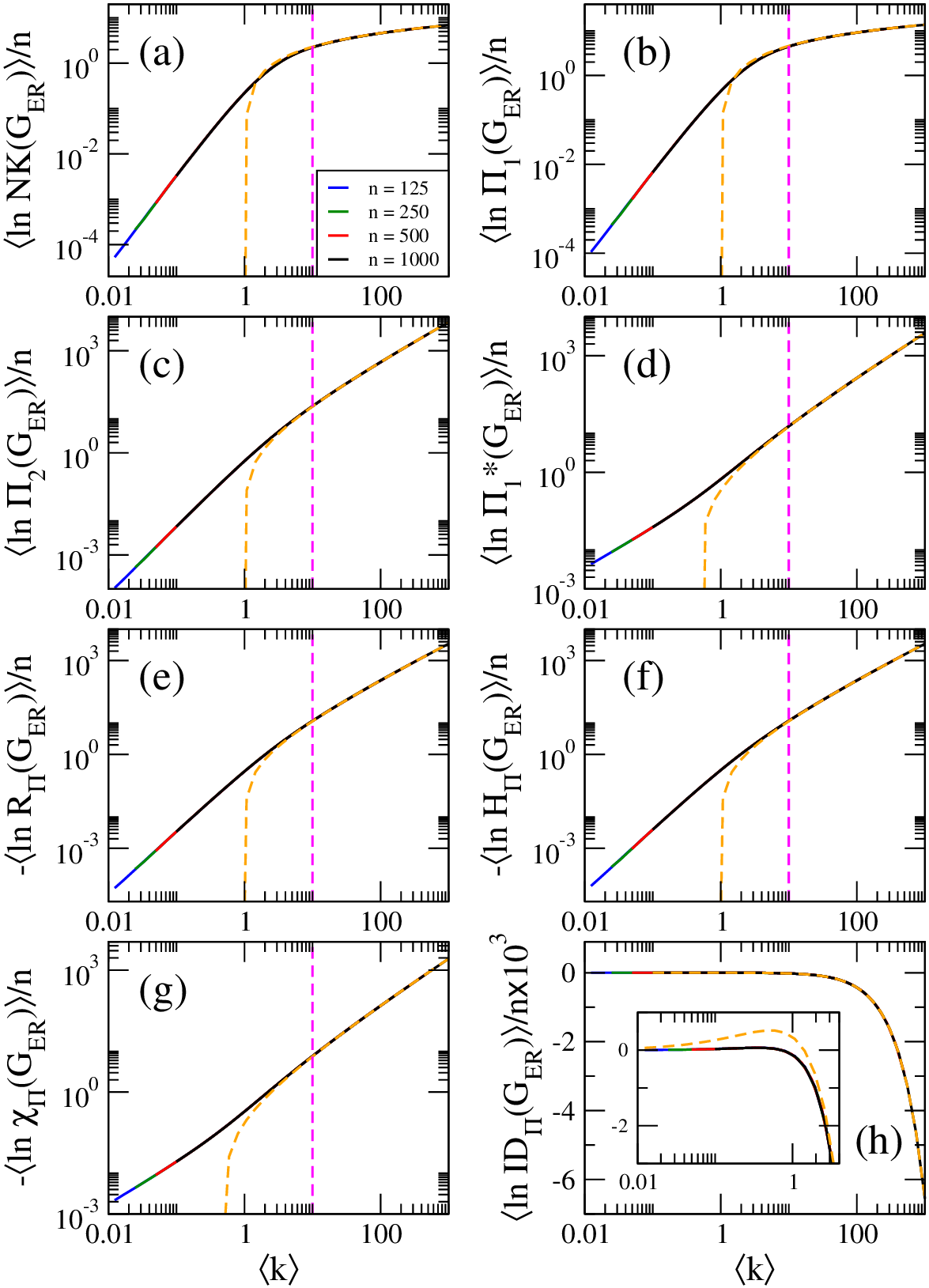}
\caption{\footnotesize{
Average logarithm of the multiplicative topological indices of Eqs.~(\ref{NK}-\ref{mID}), normalized to the
network size $n$, as a function of the average degree $\left< k \right>$ of Erd\"os-R\'enyi networks.
Orange dashed lines in (a-h) are Eqs.~(\ref{NKofGER}-\ref{mIDofGER}), respectively.
The vertical magenta dashed lines indicate $\left< k \right>=10$.
The inset in (h) is an enlargement for small $\left< k \right>$.
}}
\label{Fig02}
\end{center}
\end{figure}

Having Eqs.~(\ref{NKofGER}-\ref{mIDofGER}) as a guide, in what follows we compute the average
values of the logarithm of the MTIs listed in Eqs.~(\ref{NK}-\ref{mID}).
All averages are computed over ensembles of $10^7/n$ ER networks characterized
by the parameter pair $(n,p)$.

In Fig.~\ref{Fig01} we present the average logarithm of the eight MTIs
of Eqs.~(\ref{NK}-\ref{mID}) as a function of the probability $p$ of ER networks of sizes
$n=\{125,250,500,1000\}$. Since $\left< \ln X_\Pi(G_{\tbox{ER}}) \right><0$ for $R_\Pi$, $H_\Pi$ and $\chi_\Pi$
we conveniently plotted $-\left< \ln X_\Pi(G_{\tbox{ER}}) \right>$ in log scale to have a detailed view of
the data for small $p$, see Figs.~\ref{Fig01}(e-g).
Thus we observe that $\left< \ln X_\Pi(G_{\tbox{ER}}) \right>$ for $NK$, $\Pi_{1,2}$ and $\Pi_1^*$
[for $R_\Pi$, $H_\Pi$ and $\chi_\Pi$] is a monotonically increasing [monotonically decreasing] function of $p$.
In contrast, $\left< \ln ID_\Pi(G_{\tbox{ER}}) \right>$ is a nonmonotonic function of $p$, see Fig.~\ref{Fig01}(h).

Then, following Eqs.~(\ref{NKofGER}-\ref{mIDofGER}), in Fig.~\ref{Fig02} we show
again the average logarithm of the MTIs but now normalized to the network size,
$\left< \ln X_\Pi(G_{\tbox{ER}}) \right>/n$, as a function of $\left< k \right>$.
As can be clearly observed in this figure, the curves $\left< \ln X_\Pi(G_{\tbox{ER}}) \right>/n$ versus
$\left< k \right>$ fall one on top of the other for different network sizes; so these indices are properly
scaled where $\left< k \right>$ works as the scaling parameter. That is, $\left< k \right>$ is the
parameter that fixes the average values of the normalized MTIs on ER networks.
It is remarkable that the scaling of $\left< \ln X_\Pi(G_{\tbox{ER}}) \right>/n$ with $\left< k \right>$,
expected for $\left<  k \right>\gg 1$ according to Eqs.~(\ref{NKofGER}-\ref{mIDofGER}),
works perfectly well for all values of $\left<  k \right>$; even for $\left<  k \right>\ll 1$.
Also, in Figs.~\ref{Fig02}, we show that Eqs.~(\ref{NKofGER}-\ref{mIDofGER})
(orange-dashed lines) indeed describe well the data (thick full curves) for $\left<  k \right>\ge 10$,
which can be regarded as de dense limit of the ER model.

\begin{figure}[t!]
\begin{center}
\includegraphics[width=0.75\textwidth]{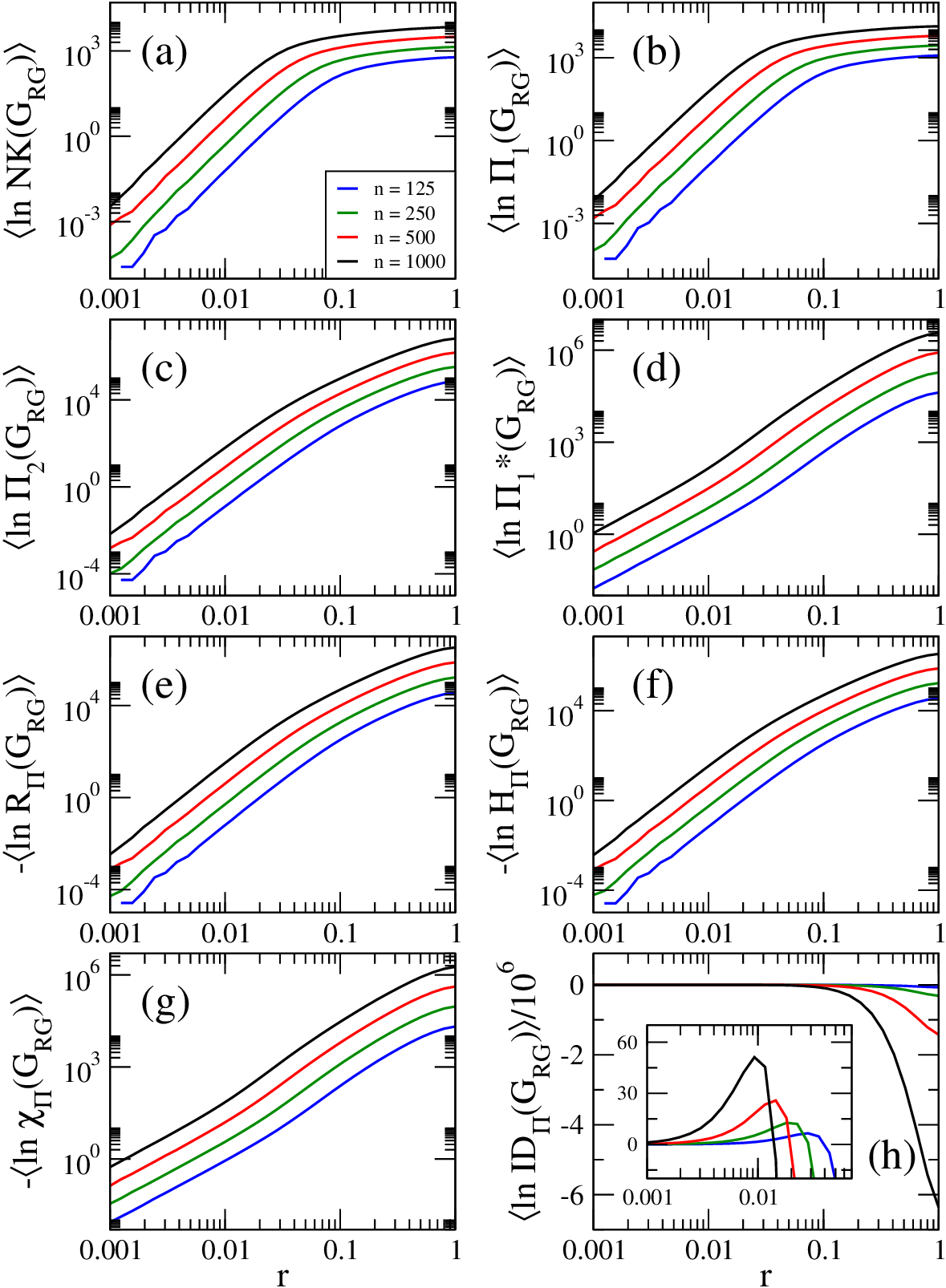}
\caption{\footnotesize{
Average logarithm of the multiplicative topological indices of Eqs.~(\ref{NK}-\ref{mID}) as a function of
the connection radius $r$ of random geometric graphs of size $n$.
The inset in (h) is an enlargement for small $r$.
}}
\label{Fig03}
\end{center}
\end{figure}
\begin{figure}[t!]
\begin{center}
\includegraphics[width=0.75\textwidth]{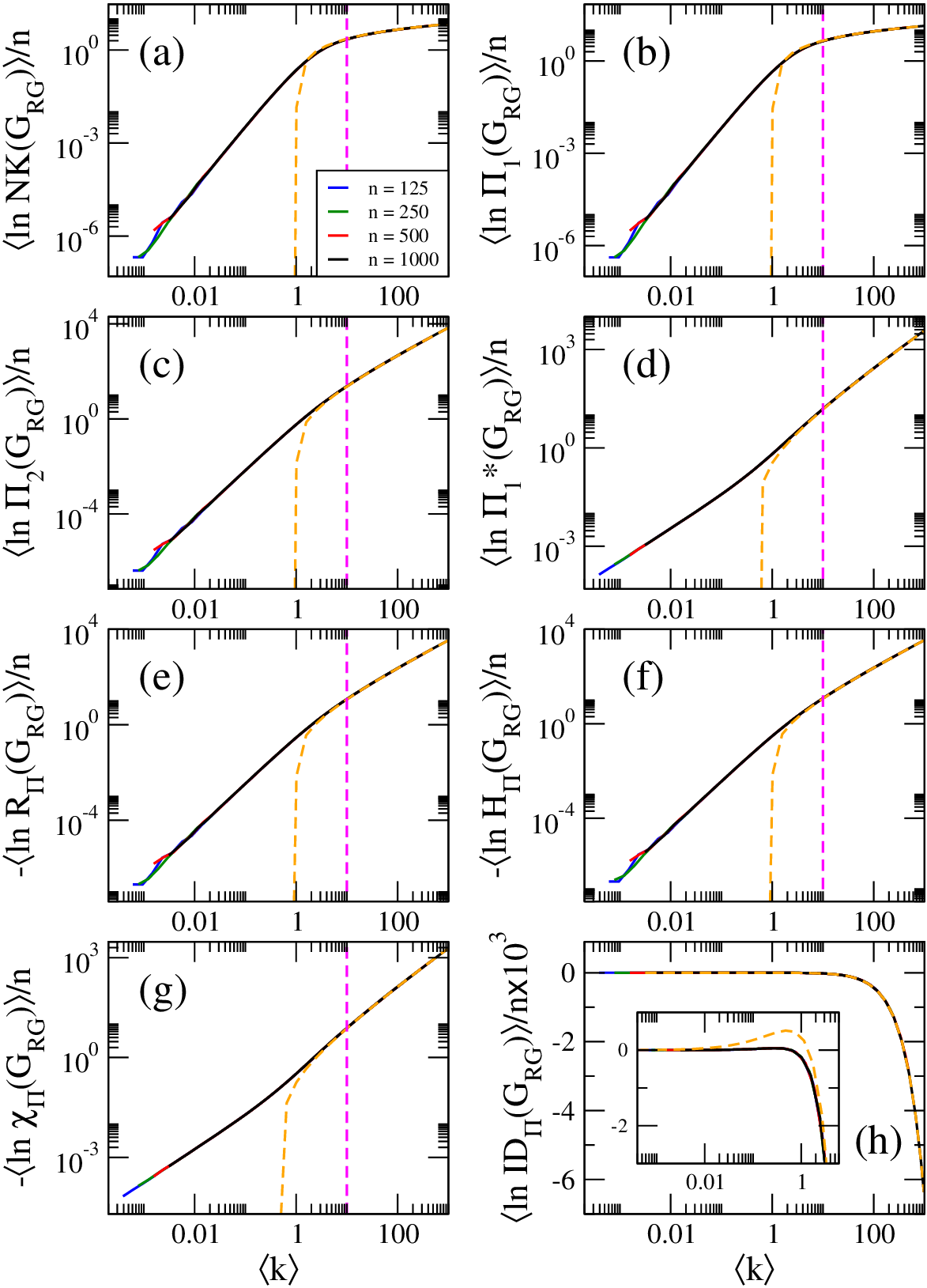}
\caption{\footnotesize{
Average logarithm of the multiplicative topological indices of Eqs.~(\ref{NK}-\ref{mID}), normalized to the
graph size $n$, as a function of the average degree $\left< k \right>$ of random geometric graphs.
Orange dashed lines in (a-h) are Eqs.~(\ref{NKofGRG}-\ref{mIDofGRG}), respectively.
The vertical magenta dashed lines indicate $\left< k \right>=10$.
The inset in (h) is an enlargement for small $\left< k \right>$.
}}
\label{Fig04}
\end{center}
\end{figure}

\subsection{Multiplicative topological indices on random geometric graphs}
\label{RG}

As in the previous Subsection, here we start by exploring the dense limit.
Indeed, for RG graphs in the dense limit, i.e. when $\left<  d \right>\gg 1$, we can approximate
$d_u \approx d_v \approx \left<  d \right>$, where
\begin{equation}
\label{kRG}
\left< d \right> = (n-1)g(r)
\end{equation}
and~\cite{EM15}
\begin{equation}
g(r) =
\left\{
\begin{array}{ll}
           r^2  \left[ \pi - \frac{8}{3}r +\frac{1}{2}r^2 \right] , & 0 \leq r \leq 1 \, ,
           \vspace{0.25cm} \\
             \frac{1}{3} - 2r^2 \left[ 1 - \arcsin(1/r) + \arccos(1/r) \right] \vspace{0.15cm} \\
             \qquad +\frac{4}{3}(2r^2+1) \sqrt{r^2-1} -\frac{1}{2}r^4 , & 1 \leq r \leq \sqrt{2} \, .
\end{array}
\right.
\label{g(r)}
\end{equation}
Thus, we obtain:
\begin{equation}
\label{NKofGRG}
\frac{\ln NK(G_{\tbox{RG}})}{n} \approx \ln \left<  d \right> ,
\end{equation}
\begin{equation}
\label{P1ofGRG}
\frac{\ln \Pi_1(G_{\tbox{RG}})}{n} \approx 2 \ln \left<  d \right> ,
\end{equation}
\begin{equation}
\label{P2ofGRG}
\frac{\ln \Pi_2(G_{\tbox{RG}})}{n} \approx \left<  d \right> \ln \left<  d \right> ,
\end{equation}
\begin{equation}
\label{P1*ofGRG}
\frac{\ln \Pi_1^*(G_{\tbox{RG}})}{n} \approx \frac{1}{2} \left<  d \right> \ln (2\left<  d \right>) ,
\end{equation}
\begin{equation}
\frac{\ln R_\Pi(G_{\tbox{RG}})}{n} \approx -\frac{1}{2} \left<  d \right> \ln \left<  d \right> ,
\label{mRofGRG}
\end{equation}
\begin{equation}
\label{mHofGRG}
\frac{\ln H_\Pi(G_{\tbox{RG}})}{n} \approx -\frac{1}{2} \left<  d \right> \ln \left<  d \right> ,
\end{equation}
\begin{equation}
\frac{\ln \chi_\Pi(G_{\tbox{RG}})}{n} \approx -\frac{\ln 2}{4} \left<  d \right>  - \frac{1}{4}\left<  d \right> \ln \left<  d \right> ,
\label{mXofGRG}
\end{equation}
and
\begin{equation}
\frac{\ln ID_\Pi(G_{\tbox{RG}})}{n} \approx \frac{\ln 2}{2} \left<  d \right>  - \left<  d \right> \ln \left<  d \right>.
\label{mIDofGRG}
\end{equation}

\begin{figure}[t!]
\begin{center}
\includegraphics[width=0.75\textwidth]{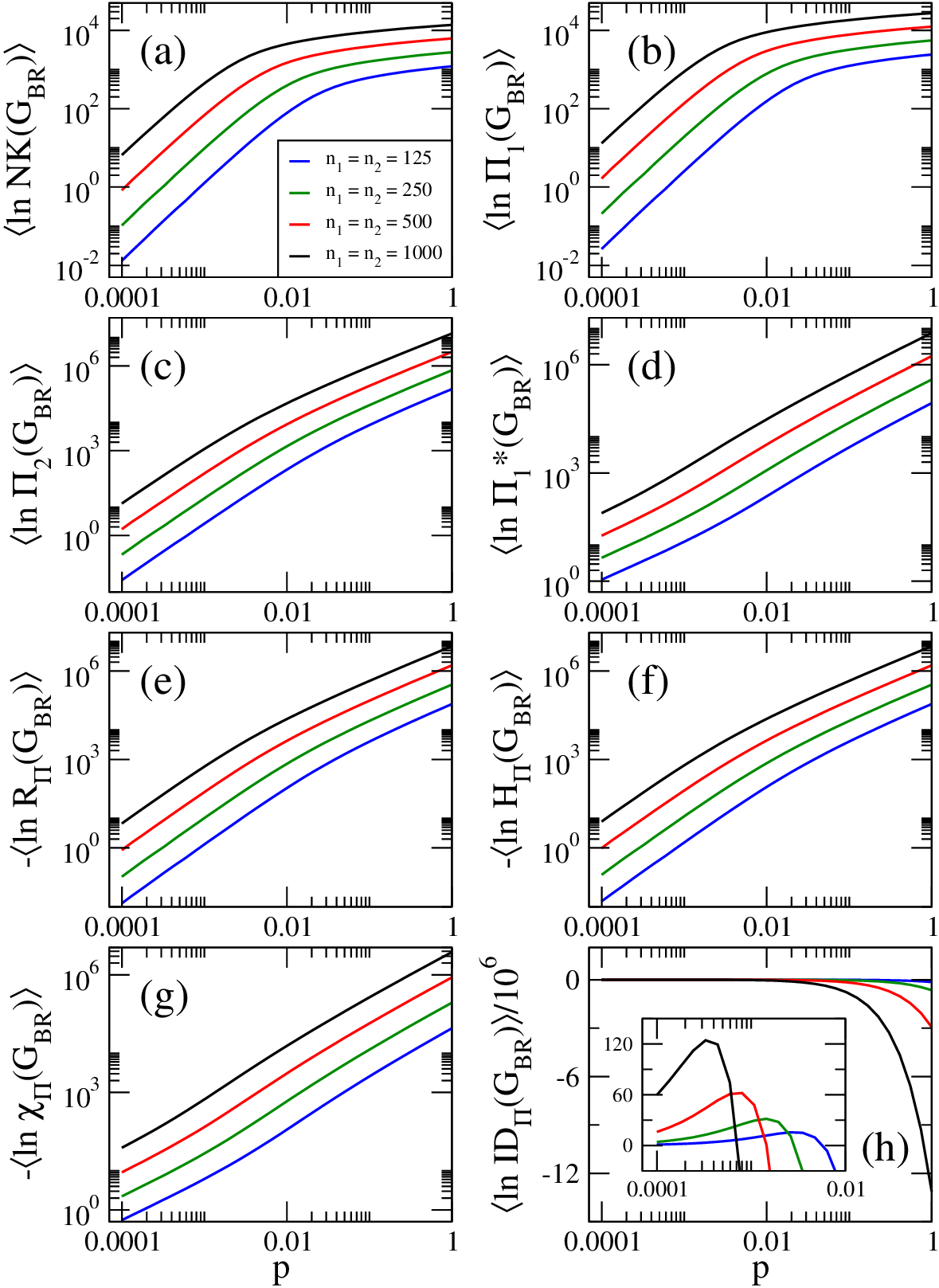}
\caption{\footnotesize{
Average logarithm of the multiplicative topological indices of Eqs.~(\ref{NK}-\ref{mID}) as a function of the
probability $p$ of bipartite random networks of sizes $n_1$ and $n_2$.
In all panels $n_1=n_2=\{125,250,500,1000\}$.
The inset in (h) is an enlargement for small $p$.
}}
\label{Fig05}
\end{center}
\end{figure}
\begin{figure}[t!]
\begin{center}
\includegraphics[width=0.75\textwidth]{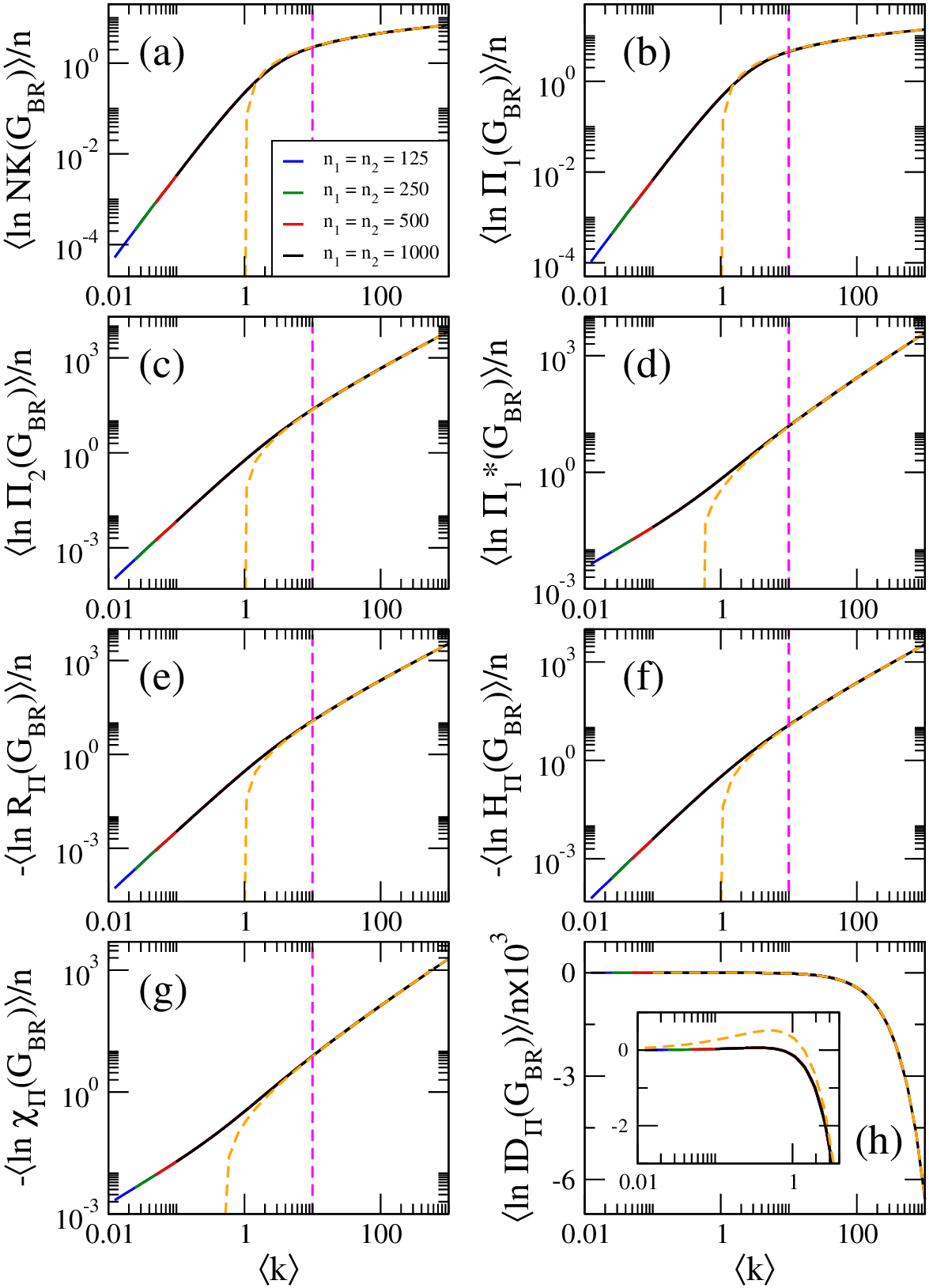}
\caption{\footnotesize{
Average logarithm of the multiplicative topological indices of Eqs.~(\ref{NK}-\ref{mID}), normalized to the
network size $n$, as a function of the average degree $\left< k \right>$ of bipartite random networks of
sizes $n_1$ and $n_2$. In all panels $n_1=n_2=\{125,250,500,1000\}$.
Orange dashed lines in (a,b) are Eqs.~(\ref{NKofGER},\ref{P1ofGER}), respectively.
Orange dashed lines in (c-h) are Eqs.~(\ref{P2ofGBR}-\ref{mIDofGBR}), respectively.
The vertical magenta dashed lines indicate $\left< k \right>=10$.
The inset in (h) is an enlargement for small $\left< k \right>$.
}}
\label{Fig06}
\end{center}
\end{figure}

Remarkably, the approximate Eqs.~(\ref{NKofGRG}-\ref{mIDofGRG}) for RG graphs are exactly the same as
the corresponding equations for ER graphs, see Eqs.~(\ref{NKofGER}-\ref{mIDofGER}); however note that the
definition of $\left<  d \right>$ is different for both models, i.e. compare Eqs.~(\ref{k}) and (\ref{kRG}-\ref{g(r)}).

Then, in Fig.~\ref{Fig03} we present the average logarithm of the eight MTIs
of Eqs.~(\ref{NK}-\ref{mID}) as a function of the connection radius $r$ of RG graphs of sizes
$n=\{125,250,500,1000\}$. All averages are computed over ensembles of $10^7/n$
random graphs, each ensemble is characterized by a fixed parameter pair $(n,r)$.

For comparison purposes, Fig.~\ref{Fig01} for ER networks and Fig.~\ref{Fig03} for RG graphs have the same
format and contents.
In fact, all the observations made in the previous Subsection for ER networks are also valid for RG graphs
by replacing $G_{\tbox{ER}}\to G_{\tbox{RG}}$ and $p\to g(r)$.
Therefore, in Fig.~\ref{Fig04} we plot $\left< \ln X_\Pi(G_{\tbox{RG}}) \right>/n$ as a function of $\left< k \right>$
showing that all curves are properly scaled; that is, curves for different graph sizes fall on top of each other.
Also, in Fig.~\ref{Fig04}, we show that Eqs.~(\ref{NKofGRG}-\ref{mIDofGRG})
(orange-dashed lines) indeed describe well the data (thick full curves) for $\left<  k \right>\ge 10$,
which can be regarded as de dense limit of RG graphs.

\subsection{Multiplicative topological indices on bipartite random networks}
\label{BR}

We start by writing approximate expressions for the MTIs on BR networks in the dense limit.
Moreover, since edges in a bipartite network join vertices of different sets, and we are labeling
here the sets as set 1 and set 2, we replace $d_u$ by $d_1$ and $d_v$ by $d_2$ in the expression
for the MTIs.
Thus, when $n_1p\gg 1$ and $n_2p\gg 1$, we can approximate $d_u = d_1 \approx \left<  d_1 \right>$
and $d_v = d_2 \approx \left<  d_2 \right>$ in Eqs.~(\ref{NK}-\ref{mID}), with
\begin{equation}
\label{kBR}
\left< d_{1,2} \right> = n_{2,1}p .
\end{equation}
Therefore, in the dense limit, the MTIs of Eqs.~(\ref{P2}-\ref{mID}) on BR networks are
well approximated by:
\begin{equation}
\label{P2ofGBR}
\frac{\ln \Pi_2(G_{\tbox{BR}})}{n_{1,2}} \approx \left<  d_{1,2} \right> \ln (\left<  d_1 \right>  \left<  d_2 \right>),
\end{equation}
\begin{equation}
\label{P1*ofGBR}
\frac{\ln \Pi_1^*(G_{\tbox{BR}})}{n_{1,2}} \approx \left<  d_{1,2} \right> \ln (\left<  d_1 \right>  + \left<  d_2 \right>) ,
\end{equation}
\begin{equation}
\frac{\ln R_\Pi(G_{\tbox{BR}})}{n_{1,2}} \approx -\frac{1}{2} \left< d_{1,2} \right> \ln \left( \left<  d_1 \right> \left<  d_2 \right> \right) ,
\label{mRofGBR}
\end{equation}
\begin{equation}
\label{mHofGBR}
\frac{\ln H_\Pi(G_{\tbox{BR}})}{n_{1,2}} \approx \left< d_{1,2} \right> \left[ \ln 2 - \ln \left( \left<  d_1 \right> + \left<  d_2 \right> \right) \right],
\end{equation}
\begin{equation}
\frac{\ln \chi_\Pi(G_{\tbox{BR}})}{n_{1,2}} \approx -\frac{1}{2} \left< d_{1,2} \right> \ln \left( \left<  d_1 \right> + \left<  d_2 \right> \right) ,
\label{mXofGBR}
\end{equation}
and
\begin{equation}
\frac{\ln ID_\Pi(G_{\tbox{BR}})}{n_{1,2}} \approx \left< d_{1,2} \right> \ln \left( \frac{1}{\left<  d_1 \right>^2} + \frac{1}{\left<  d_2 \right>^2} \right) .
\label{mIDofGBR}
\end{equation}
Above we used $\mid E(G_{\tbox{BR}}) \mid=n_1n_2p=n_{1,2}\left<  d_{1,2} \right>$.
We note that we are not providing approximate expressions for $NK(G_{\tbox{BR}})$ nor $\Pi_1(G_{\tbox{BR}})$.

Now we compute MTIs on ensembles of $10^7/n$ BR networks. In contrast to
ER and RG networks now the BR network ensembles are characterized by three parameters: $n_1$, $n_2$, and $p$.
For simplicity, but without lost of generality, in our numerical calculations we consider $n_1=n_2$.
It is remarkable to notice that in the case of $n_1=n_2=n/2$, where $\left< k_1 \right>=\left< k_2 \right>=\left< k \right>=np/2$,
Eqs.~(\ref{P2ofGBR}-\ref{mIDofGBR}) reproduce Eqs.~(\ref{P2ofGER}-\ref{mIDofGER}).

Then, in Fig.~\ref{Fig05} we present the average logarithm of the eight MTIs
of Eqs.~(\ref{NK}-\ref{mID}) as a function of the probability $p$ of BR networks of sizes
$n_1=n_2=\{125,250,500,1000\}$.
Therefore, by plotting $\left< \ln X_\Pi(G_{\tbox{BR}}) \right>/n$ as a function of $\left< k \right>$,
see Fig.~\ref{Fig06}, we confirm that the curves of the MTIs on BR networks are properly scaled,
as predicted by Eqs.~(\ref{P2ofGBR}-\ref{mIDofGBR}); see the orange dashed lines in Fig.~\ref{Fig06}(c-h).
Here, we can also say that $\left<  k \right>\ge 10$ can be regarded as de dense limit of BR networks.
We have verified (not shown here) that we arrive to equivalent conclusions when $n_1\neq n_2$.

Note that in panels (a,b) of Fig.~\ref{Fig06} we included Eqs.~(\ref{NKofGER},\ref{P1ofGER}) as
orange dashed lines. Those equations were obtained for ER networks, however they reproduce
perfectly well both $NK(G_{\tbox{BR}})$ and $\Pi_1(G_{\tbox{BR}})$ when $\left<  k \right>\ge 10$
since we are considering $n_1=n_2=n/2$.

\section{Conclusions}
\label{Conclusions}

In this work we have performed a thorough analytical and statistical study of multiplicative topological indices (MTIs)
on random networks.
As models of random networks we have used
Erd\"os-R\'enyi networks, random geometric graphs, and bipartite random networks.

From an analytical viewpoint, we find several inequalities that relate multiplicative indices to their additive versions.

Within a statistical approach, we show that the average logarithm of the MTIs,
$\left< \ln X_\Pi(G) \right>$, normalized to the order of the network, scale with
the network average degree $\left< k \right>$.
Thus, we conclude that $\left< k \right>$ is the parameter that fixes the average values of the
logarithm of the MTIs on random networks. Moreover, the equivalence among
Eqs.~(\ref{NKofGER}-\ref{mIDofGER}) for Erd\"os-R\'enyi networks,
Eqs.~(\ref{NKofGRG}-\ref{mIDofGRG}) for random geometric graphs, and
Eqs.~(\ref{P2ofGBR}-\ref{mIDofGBR}) for bipartite random networks (when the subsets are equal in size)
allows us to state a scaling law that connects the three graph models.
That is, given the multiplicative topological index $X_\Pi$, the average of its logarithm divided by the
network size is the same function of the average degree regardless the network model:
\begin{equation}
\frac{\left< \ln X_\Pi(G_{\tbox{ER}}) \right>}{n} \approx \frac{\left< \ln X_\Pi(G_{\tbox{RG}}) \right>}{n} \approx \frac{\left< \ln X_\Pi(G_{\tbox{BR}}) \right>}{n}
\approx f(\left< k \right>) .
\label{fofk}
\end{equation}
To validate Eq.~(\ref{fofk}) we choose the Narumi-Katayama index, see Eq.~(\ref{NK}),
apply it to Erd\"os-R\'enyi networks, random geometric graphs and bipartite random networks and plot
$\left< \ln NK_\Pi(G) \right>/n$ in Fig.~\ref{Fig07}(a). There, we can clearly see that the curves corresponding
to the three network models fall one on top of the other, thus validating Eq.~(\ref{fofk}). Also notice that we are
using networks of different sizes to stress the scaling law in Eq.~(\ref{fofk}).
We observe the same scaling behavior in all the other MTIs of Eqs.~(\ref{P1}-\ref{mID}), as can be seen
in Figs.~\ref{Fig07}(b,c) where we present $\left< \ln \chi_\Pi(G) \right>/n$ and $\left< \ln ID_\Pi(G) \right>/n$,
respectively.

\begin{figure}[t!]
\begin{center}
\includegraphics[width=0.95\textwidth]{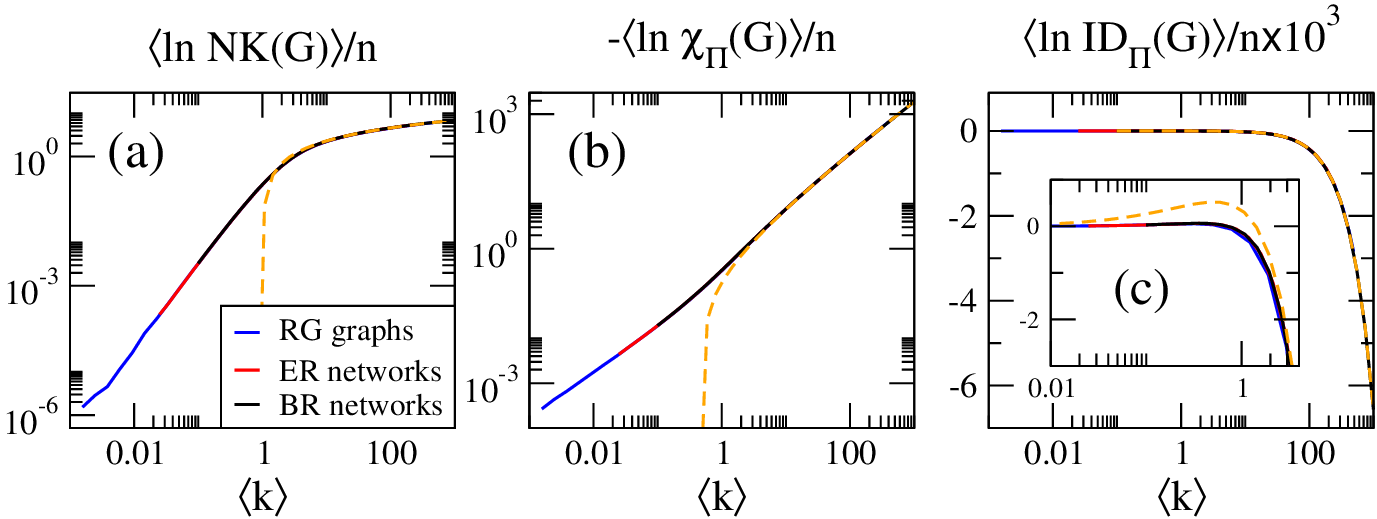}
\caption{\footnotesize{
Average logarithm of (a) $NK(G)$, (b) $\chi_\Pi(G)$, and (c) $ID_\Pi(G)$, normalized to the
network size $n$, as a function of the average degree $\left< k \right>$ of random geometric
graphs of size $n=500$, Erd\"os-R\'enyi networks of size $n=250$ and bipartite random networks
of size $n=2n_1=2n_2=2000$.
Orange dashed lines are (a) Eq.~(\ref{NKofGER}), (b) Eq.~(\ref{mXofGER}), and (c) Eq.~(\ref{mIDofGER}).
The inset in (c) is an enlargement for small $\left< k \right>$.
}}
\label{Fig07}
\end{center}
\end{figure}

Finally, it is fair to mention that we found that the multiplicative version of the geometric-arithmetic index
$GA_\Pi(G)$ on random networks does not scale with the average degree. This situation
is quite unexpected for us since all other MTIs we study here scale with $\left< k \right>$.
In addition, our previous statistical and theoretical studies of topological indices of the form
$X_\Sigma(G)$, see Eq.~(\ref{TI}), showed that the normalized geometric-arithmetic index $GA_\Sigma(G)$
scales with $\left< k \right>$~\cite{AHMS20}.
Thus, we believe that the multiplicative index $GA_\Pi(G)$ on random graphs requieres further analysis.

We hope that our work may motivate further analytical as well as computational studies of
multiplicative topological indices on random networks.

\bmhead{Acknowledgments}

The research of J.M.R. and J.M.S. was supported by a grant from Agencia Estatal de Investigaci\'on (PID2019-106433GBI00/AEI/10.13039/501100011033), Spain.
J.M.R. was supported by the Madrid Government (Comunidad de Madrid-Spain) under the Multiannual Agreement with UC3M in the line of Excellence of University Professors (EPUC3M23), and in the context of the V PRICIT (Regional Programme of Research and Technological Innovation).

\section*{Declarations}

The authors declare no conflict of interest. The founding sponsors had no role in the design of the study; in the collection, analyses, or interpretation of data; in the writing of the manuscript, and in the decision to publish the results.


\begin{thebibliography}{999}

\bibitem{G13}
I. Gutman,
Degree-based topological indices,
Croat. Chem. Acta {\bf 86}, 351--361 (2013).

\bibitem{G72}
I. Gutman and N. Trinajsti\'c,
Chem. Phys. Lett. {\bf 17} (1972) 535.

\bibitem{R75}
Randi\'c, M.
On characterization of molecular branching.
{\em J. Am. Chem. Soc.} {\bf 1975}, {\em 97} 6609--6615.

\bibitem{F87}
S. Fajtlowicz,
On conjectures of Graffiti--II,
{\it Congr. Numer.} {\bf 60} (1987) 187--197.

\bibitem{ZT46}
B. Zhou and N. Trinajsti\'c,
J. Math. Chem. {\bf 46} (2009) 1252.

\bibitem{MMRS20}
C. T. Mart\'{\i}nez-Mart\'{\i}nez, J. A. Mendez-Ber\'mudez, J. M. Rodr\'{\i}guez, and J. M. Sigarreta,
Computational and analytical studies of the Randi\'c index in Erd\"os--R\'enyi models,
Appl. Math. Comput. {\bf 377}, 125137 (2020).

\bibitem{AHMS20}
R. Aguilar-Sanchez, I. F. Herrera-Gonzalez, J. A. Mendez-Bermudez, and J. M. Sigarreta,
Computational properties of general indices on random networks.
Symmetry {\bf 12}, 1341 (2020).

\bibitem{MMRS21}
C. T. Mart\'{\i}nez-Mart\'{\i}nez, J. A. Mendez-Bermudez, J. M. Rodr\'{\i}guez, and J. M. Sigarreta,
Computational and analytical studies of the harmonic index in Erd\"os--R\'enyi models,
MATCH Commun. Math. Comput. Chem. {\bf 85}, 395--426 (2021).

\bibitem{AMRS20}
R. Aguilar-Sanchez, J. A. Mendez-Bermudez, F. A. Rodrigues, and J. M. Sigarreta-Almira,
Topological versus spectral properties of random geometric graphs,
Phys. Rev. E {\bf 102}, 042306 (2020).

\bibitem{NK84}
H. Narumi and M. Katayama,
Mem. Fac. Engin. Hokkaido Univ. {\bf 16} (1984) 209.

\bibitem{TC1084}
R. Todeschini and V. Consonni,
MATCH Commun. Math. Comput. Chem. {\bf 64} (2010) 359.

\bibitem{SR51}
R. Solomonoff and A. Rapoport,
Connectivity of random nets.
Bull. Math. Biophys. {\bf 13}, 107--117 (1951).

\bibitem{ER59}
P. Erd\"os and A. R\'enyi,
On random graphs.
Publ. Math. (Debrecen) {\bf 6}, 290--297 (1959).

\bibitem{ER60}
P. Erd\"os and A. R\'enyi,
On the evolution of random graphs,
Inst. of the Hung. Acad. of Sci. {\bf 5}, 17--61 (1960);
P. Erd\"os and A. R\'enyi,
On the strength of connectedness of a random graph,
Acta Mathematica Hungarica {\bf 12}, 261--267 (1961).

\bibitem{DC02}
J. Dall and M. Christensen,
Random geometric graphs,
Phys. Rev. E {\bf 66}, 016121 (2002).

\bibitem{P03}
M. Penrose,
Random Geometric Graphs;
(Oxford University Press, Oxford, 2003).

%

\bibitem{EM15}
E. Estrada and M. Sheerin,
Random rectangular graphs,
Phys Rev. E {\bf 91}, 042805 (2015).

\bibitem{BQRS} P. Bosch, Y. Quintana, J. M. Rodr\'{\i}guez, J. M. Sigarreta,
Jensen-type inequalities for m-convex functions, to appear in {\it Open Math.}

\bibitem{GutmanMilovanovic} I. Gutman, I. Milovanovi\'c, E. Milovanovi\'c,
Relations between Ordinary and Multiplicative Degree-Based Topological Indices,
{\it Filomat} {\bf 32:8} (2018) 3031--3042.

\bibitem{Kober} H. Kober, On the arithmetic and geometric means and on H\"older's inequality, {\it Proc. Amer. Math. Soc.} {\bf 9} (1958) 452--459.

\bibitem{Petrovic} M. Petrovi\'c's, Sur Une Fonctionnelle, {\it Publ. Math. Univ. Belgrade} {\bf 1} (1932) 146--149.

\bibitem{RetiGutman} T. R\'eti, I. Gutman, Relations between ordinary and multiplicative Zagreb indices,
{\it Bull. Inter. Math. Virtual Inst.} {\bf 2} (2012) 133--140.

\bibitem{ZGA} B. Zhou, I. Gutman, T. Aleksi\'c, A note on Laplacian energy of graphs, {\it MATCH Commun. Math. Comput. Chem.} {\bf 60} (2008) 441--446.
    

\end{thebibliography}


\end{document}